\documentclass[]{interact}

\usepackage{epstopdf}% To incorporate .eps illustrations using PDFLaTeX, etc.
\usepackage[caption=false]{subfig}% Support for small, `sub' figures and tables

\usepackage[numbers,sort&compress]{natbib}% Citation support using natbib.sty
\bibpunct[, ]{[}{]}{,}{n}{,}{,}% Citation support using natbib.sty
\renewcommand\bibfont{\fontsize{10}{12}\selectfont}% Bibliography support using natbib.sty

\usepackage{xcolor}
\usepackage{bbm}
\DeclareMathOperator{\sign}{sign}

\theoremstyle{plain}% Theorem-like structures provided by amsthm.sty
\newtheorem{theorem}{Theorem}[section]

\theoremstyle{definition}
\newtheorem{definition}[theorem]{Definition}

\theoremstyle{remark}
\newtheorem{remark}{Remark}

\begin{document}

%\articletype{}% Specify the article type or omit as appropriate

\title{Standard and fractional reflected Ornstein-Uhlenbeck processes as the limits of square roots of Cox-Ingersoll-Ross processes}

\author{
    \name{Yuliya Mishura\textsuperscript{a} and Anton Yurchenko-Tytarenko\textsuperscript{b}}\thanks{CONTACT A. Yurchenko-Tytarenko. Email: antony@math.uio.no}
    \affil{\textsuperscript{a}Department of Probability, Statistics and Actuarial Mathematics, Taras Shevchenko National University of Kyiv, 64/13, Volodymyrs'ka St., Kyiv 01601, Ukraine;\\
    \textsuperscript{b}Department of Mathematics, University of Oslo, Moltke Moes vei 35, 0851 Oslo, Norway}
}

\maketitle

\begin{abstract}
    In this paper, we establish a new connection between Cox-Ingersoll-Ross (CIR) and reflected Ornstein-Uhlenbeck (ROU) models driven by either a standard Wiener process or a fractional Brownian motion with $H>\frac{1}{2}$. We prove that,  with probability 1, the square root of the CIR process converges uniformly on compacts to the ROU process as the mean reversion parameter tends to either $\sigma^2/4$ (in the standard case) or to $0$ (in the fractional case). This also allows to obtain a new representation of the reflection function of the ROU as the limit of integral functionals of the CIR processes. The results of the paper are illustrated by simulations.     
\end{abstract}


\section{Introduction}
    Both the reflected Ornstein-Uhlenbeck (ROU) and the Cox-Ingersoll-Ross (CIR) processes are extremely popular models in a variety of fields. Without attempting to give a complete overview of possible applications due to the large amount of literature on the topic, we only mention that the ROU process is widely used in queueing theory \cite{GNR1986, WG2003-1, WG2003-2, WG2005}, in population dynamics modeling \cite{AG2004, Ricciardi1986}, in economics and finance for modeling regulated markets \cite{Krugman1991, YLRWB2012, BTWY2011, BWY2011}, interest rates \cite{GK1997} and stochastic volatility \cite{SZh1999} (see also \cite{GNdiC2012, Linetsky2005} and references therein for a more details on applications of the ROU in various fields) while the most notable usages of the CIR process are related to representing the dynamics of interest rates \cite{CIR1981, CIR1985-1, CIR1985-2} and stochastic volatility in the Heston model \cite{Heston1993}. 
    
    It is well-known \cite{Maghsoodi1996, ShW1973} that the CIR process has strong links with the standard OU dynamics; in particular, if $B = (B_1,...,B_d)$ is a $d$-dimensional Brownian motion and $U = (U_1,...,U_d)$ is a standard $d$-dimensional OU process given by
    \[
        U_i(t) = U_i(0) - \frac{b}{2}\int_0^t U_i(s)ds + \frac{\sigma}{2} B_i(t), \quad t\ge 0,\quad i=1,...,d,
    \]
    then it is easy to see via It\^o's formula that the process ${\sum_{i=1}^d U^2_i(t)}$, $t\ge 0$, is the CIR process of the form
    \begin{equation}\label{eq: CIR Introduction}
        X(t) = X(0) + \int_0^t \left(a - bX(s)\right)ds + \sigma \int_0^t \sqrt{X(s)} dW(s), \quad t\ge 0,
    \end{equation}
    with $a = \frac{d\sigma^2}{4}$ and  $W(t) := \sum_{i=1}^d \int_0^t \frac{U_i(s)}{\sqrt{\sum_{j=1}^d U^2_j(s)}}dB_i(s)$ (which is a standard Brownian motion by Levy's characterization). The value $d=\frac{4a}{\sigma^2}$ is sometimes referred to as a \emph{dimension} or a \emph{number of degrees of freedom} of the CIR process (see e.g. \cite{Maghsoodi1996} and references therein) and thus, in this terminology, a square of a standard one-dimensional OU process turns out to be a CIR process with one degree of freedom w.r.t. another Brownian motion. 
    
    In this paper, we investigate a connection between the CIR and the ROU processes that is in some sense related to the one described above. Namely, in the first part we prove that the ROU process
    \begin{equation}\label{eq: ROU Introduction}
        Y(t) = Y(0) - \frac{b}{2} \int_0^t Y(s) ds + \frac{\sigma}{2} W(t) + L(t),\quad t\ge 0,
    \end{equation}
    where $W$ is a standard Brownian motion and $L$ is a continuous non-decreasing process that can have points of growth only at zeros of $Y$, coincides with the square root of the CIR process of the type \eqref{eq: CIR Introduction} with $a=\frac{\sigma^2}{4}$ (i.e. with one degree of freedom) driven by the same Brownian motion $W$. Moreover, if $\{\varepsilon_n,~n\ge 1\}$ is a sequence of positive numbers such that $\varepsilon_n \downarrow 0$ as $n\to\infty$, then, with probability 1, for all $T>0$
    \begin{equation}\label{eq: representation introduction}
        \sup_{t\in [0,T]}\left| L(t) - \frac{1}{2} \int_0^t \frac{\varepsilon_n}{\sqrt{X_{\varepsilon_n}(s)}}ds\right| \to 0, \quad n\to\infty,
    \end{equation}
    where $X_{\varepsilon_n}$ is the CIR process of the form
    \[
        X_{\varepsilon_n}(t) = X(0) + \int_0^t \left(\frac{\sigma^2}{4} + \varepsilon_n - bX_{\varepsilon_n}(s)\right)ds + \sigma \int_0^t \sqrt{X_{\varepsilon_n}(s)} dW(s).
    \]
    
    The second part of the paper discusses the connection between fractional counterparts of equations \eqref{eq: CIR Introduction} and \eqref{eq: ROU Introduction} driven by fractional Brownian motion $\{B^H(t),~t\ge 0\}$ with Hurst index $H>\frac{1}{2}$. Namely, we consider a fractional Cox-Ingersoll-Ross process
    \[
        X_\varepsilon^H(t) = X(0) + \int_0^t \left(\varepsilon - b X_\varepsilon^H(s)\right)ds + \sigma \int_0^t \sqrt{X_\varepsilon^H(s)} dB^H(s), \quad t\ge 0,
    \]
    where the integral $\int_0^t \sqrt{X^H(s)} dB^H(s)$ is understood as the pathwise limit of Riemann-Stieltjes integral sums (see \cite{MYuT2018} or \cite[Subsection 4.1]{DNMYT2020}) and prove that with probability 1 the paths of $\{\sqrt{X_\varepsilon^H(t)},~t\ge 0\}$ a.s. converge to the reflected fractional Ornstein-Uhlenbeck (RFOU) process uniformly on each compact $[0,T]$ as $\varepsilon \downarrow 0$. Moreover, an analogue of the representation \eqref{eq: representation introduction} also takes place: if $L^H$ is a reflection function of the RFOU process, then, with probability 1, for each $T>0$
    \[
        \sup_{t\in[0,T]}\left| L^H(t) - \frac{1}{2} \int_0^t \frac{\varepsilon}{\sqrt{X^H_\varepsilon (s)}}ds\right| \to 0, \quad \varepsilon \downarrow 0.
    \]
    
    The paper is organised as follows. In section \ref{sec: standard case}, we consider the link between the CIR and the ROU processes in the standard Wiener case. Section \ref{sec: fractional case} is devoted to the fractional setting. Section \ref{sec: simulations} contains simulations that illustrate our results.

\section{Classical reflected Ornstein-Uhlenbeck and Cox-Ingersoll-Ross processes}\label{sec: standard case}

The main goal of this section is to establish connection between Cox-Ingersoll-Ross (CIR) and reflected Ornstein-Uhlenbeck (ROU) processes in the standard Brownian setting. We shall start from the definition of a \emph{reflection function} following the one given in the classical work \cite{Skorokhod1961}.

\begin{definition}\label{def: reflection function}
    Let $\xi = \{\xi(t),~t\ge 0\}$ be some stochastic process. The process $\zeta = \{\zeta(t),~t\ge 0\}$ is called \emph{a reflection function} for $\xi$, if $\zeta$ is, with probability 1, a continuous non-decreasing process such that $\zeta(0)=0$ and the points of growth of $\zeta$ can occur only at zeros of $\xi$.
\end{definition}

\begin{definition}
    Stochastic process $\widetilde Y = \{\widetilde Y(t),~t\ge 0\}$ is called a \emph{reflected Ornshein-Uhlenbeck (ROU) process} if it satisfies a stochastic differential equation of the form
    \begin{equation}\label{eq: reflected OU process equation}
        \widetilde Y(t) = Y(0) - \widetilde{b} \int_0^t \widetilde Y(s) ds + \widetilde{\sigma} W(t) + \widetilde L(t),\quad t\ge 0,
    \end{equation}
    where $Y(0)$, $\widetilde{b}$ and $\widetilde{\sigma}$ are positive constants, $W = \{W(t),~t\ge 0\}$ is a standard Brownian motion, $\{\widetilde L(t),~t \ge 0\}$ is a reflection function for $\widetilde Y$ and $\widetilde Y \ge 0$ a.s.
\end{definition}

\begin{remark}\label{rem: uniqueness of reflection function}
    The ROU process is well-known and studied in the literature, see e.g. \cite{WG2003-2} and references therein. Note also that, despite \eqref{eq: reflected OU process equation} has two unknown functions $\widetilde Y$ and $\widetilde L$, the solution is still unique. Indeed, let $\widetilde Y$ and $\widehat Y$ be two stochastic processes satisfying
    \[
        \widetilde Y(t) = Y(0) - \widetilde{b} \int_0^t \widetilde Y(s) ds + \widetilde{\sigma} W(t) + \widetilde L(t)
    \]
    and
    \[
        \widehat{Y}(t) = Y(0) - \widetilde b \int_0^t \widehat {Y}(s) ds + \widetilde{\sigma} W(t) + \widehat L(t),
    \]
    where $\widetilde L$ and $\widehat L$ are the corresponding reflection functions. Assume that on some $\omega\in\Omega$ such that both $\widetilde Y$ and $\widehat Y$ are continuous
    \begin{equation}\label{proofeq: assume that the difference is positive}
        \widetilde{Y}(t) - \widehat Y(t) > 0
    \end{equation}
    and consider $\tau(t) := \sup\{s\in[0,t):~\widetilde{Y}(t) - \widehat Y(t) = 0\}$. Then $\widetilde{Y}(u) - \widehat Y(u) > 0$ for all $u\in(\tau(t), t]$; moreover, $\widetilde Y(u) > 0$ for $u\in(\tau(t), t]$, therefore $\widetilde L$ is non-increasing on $(\tau(t), t]$. It means that the difference $\widetilde{Y}(u) - \widehat Y(u)$ is also non-increasing on $(\tau(t), t]$ since
    \[
        \widetilde{Y}(u) - \widehat Y(u) = - \widetilde{b} \int_{\tau(t)}^u (\widetilde{Y}(s) - \widehat Y(s)) ds + \left(\widetilde L(u) - \widehat L(u)\right) - \left(\widetilde L(\tau(t)) - \widehat L(\tau(t))\right)
    \]
    and the right-hand side is non-increasing w.r.t. $u$. Whence, taking into account that $\widetilde{Y}({\tau(t)}) - \widehat Y({\tau(t)}) = 0$ due to the definition of $\tau(t)$ and continuity of both $\widetilde Y$ and $\widehat Y$, the difference $\widetilde{Y}(u) - \widehat Y(u)$ cannot be positive for any $u\in(\tau(t), t]$ which contradicts \eqref{proofeq: assume that the difference is positive}. Interchanging the roles of $\widetilde{Y}$ and $\widehat Y$, one can easily verify that $\widetilde{Y}(t) - \widehat Y(t)$ cannot be negative either and whence $\widehat Y = \widetilde Y$, $\widehat L = \widetilde L$.
\end{remark} 

Now, consider a standard CIR process defined as a continuous modification of the unique solution to the equation
\begin{equation}\label{0.q}
   X(t) = X(0) + \int_0^t \left(a - bX(s)\right)ds + \sigma \int_0^t \sqrt{X(s)} dW(s), \quad t\ge 0,
\end{equation}
where $X(0), a, b, \sigma > 0$ and $W = \{W(t),~t\ge 0\}$ is a classical Wiener process. It is well-known (see e.g. \cite[Example 8.2]{IW}) that for $a>0$ the solution $\{X(t),~t\ge 0\}$ is non-negative a.s. for any $t\ge 0$; moreover, the solution is strictly positive a.s. provided that $a\ge \frac{\sigma^2}{2}$, see e.g. \cite[Chapter 5]{KS}. Therefore, if $a>0$, the square-root process $Y = \{Y(t),~t\in[0,T]\} := \{\sqrt{X(t)},~t\in[0,T]\}$ is well-defined.

For an arbitrary $\varepsilon >0$, consider a the stochastic process $\{\sqrt{X(t) + \varepsilon},~t \in [0,T]\}$. By It\^o's formula, for any $t\ge 0$
\begin{equation}\label{eq: Ito}
\begin{aligned}
    \sqrt{X(t) + \varepsilon} &= \sqrt{X(0) + \varepsilon} + \frac{1}{2} \int_0^t \left( \frac{a}{\sqrt{X(s) + \varepsilon}} - \frac{\sigma^2}{4} \frac{X(s)}{(X(s) + \varepsilon)^{\frac{3}{2}}} \right)ds 
    \\
    &\quad - \frac{1}{2} \int_0^t \frac{b X(s)}{\sqrt{X(s) + \varepsilon}} ds + \frac{\sigma}{2} \int_0^t \frac{\sqrt{X(s)}}{\sqrt{X(s) + \varepsilon}} dW(s)
\end{aligned}
\end{equation}
and, since the left-hand side of \eqref{eq: Ito} converges to $\sqrt{X(t)} = Y(t)$ a.s. as $\varepsilon \to 0$, moving $\varepsilon \to 0$ in the right-hand side would give us the dynamics of $Y$. 

First, it is clear that for any $t\ge 0$
\begin{equation}\label{eq: convergence CIR 1}
\begin{gathered}
    \sqrt{X(0) + \varepsilon} \to Y(0) \quad a.s.  
\end{gathered}
\end{equation}
and
\begin{equation}\label{eq: convergence CIR 2}
\begin{gathered}
    \int_0^t \frac{X(s)}{\sqrt{X(s) + \varepsilon}} ds \to \int_0^t Y(s) ds \quad a.s. 
\end{gathered}
\end{equation}
as $\varepsilon \to 0$. Further, by the monotone convergence,
\begin{equation}\label{eq: place where integral of 1 over Y arises}
\begin{gathered}
    \int_0^t \frac{1}{\sqrt{X(s) + \varepsilon}} ds \to  \int_0^t \frac{1}{Y(s)} ds \in [0,\infty)\cup\{\infty\} \quad a.s.,
    \\
    \int_0^t \frac{X(s)}{(X(s) + \varepsilon)^{\frac{3}{2}}} \to \int_0^t \frac{1}{Y(s)} ds \in [0,\infty)\cup\{\infty\} \quad a.s.
\end{gathered}
\end{equation}
as $\varepsilon \to 0$. Finally, by Burkholder-Davis-Gundy inequality and dominated convergence theorem, for any $T>0$
\begin{equation*}
\begin{gathered}
    \mathbb E\left(\sup_{t\in [0,T]}\left|\int_0^t \frac{\sqrt{X(s)}}{\sqrt{X(s) + \varepsilon}} dW(s) - W(t)\right|\right)^{2} \le 4\mathbb E \int_0^T \left(\frac{\sqrt{X(s)}}{\sqrt{X(s) + \varepsilon}} -1\right)^{2} ds
    \\
    = 4\mathbb E\int_0^T \left(\frac{\sqrt{X(s)}}{\sqrt{X(s) + \varepsilon}} -1\right)^{2} \mathbbm{1}_{\{X(s) > 0 \}} ds + 4 \mathbb{E}\int_0^T  \mathbbm{1}_{\{X(s) = 0 \}} ds
    \\
    = 4\mathbb E\int_0^T \left(\frac{\sqrt{X(s)}}{\sqrt{X(s) + \varepsilon}} -1\right)^{2} \mathbbm{1}_{\{X(s) > 0 \}} ds \to 0, \quad \varepsilon \to 0,
\end{gathered}
\end{equation*}
where we used continuity of the distribution of $X(s)$ for each $s>0$ to state that $4 \mathbb{E}\int_0^T  \mathbbm{1}_{\{X(s) = 0 \}} ds = 0$ (see e.g. \cite{Maghsoodi1996} and references therein). This implies that
\begin{equation}\label{eq: L2 convergence of stochastic integral to W}
    \sup_{t\in [0,T]}\left|\int_0^t \frac{\sqrt{X(s)}}{\sqrt{X(s) + \varepsilon}} dW(s) - W(t)\right| \xrightarrow{L^2(\Omega)} 0, \qquad \varepsilon \to 0.
\end{equation}
By \eqref{eq: L2 convergence of stochastic integral to W}, it is evident that there exists a sequence $\{\varepsilon_{n},~n \ge 1\}$ which depends on $T$ such that
\begin{equation}\label{eq: a.s.}
\begin{gathered}
    \sup_{t\in[0,T]}\left|\int_0^t \frac{\sqrt{X(s)}}{\sqrt{X(s) + \varepsilon_{n}}} dW(s) - W(t)\right| \to 0 \quad a.s., \quad n\to\infty
\end{gathered}
\end{equation}
and along this sequence 
\begin{equation}\label{eq: the reflected piece a.s.}
    \lim_{n\to\infty} \frac{1}{2} \int_0^t \left( \frac{a}{\sqrt{X(s) + \varepsilon_n}} - \frac{\sigma^2}{4}\frac{X(s)}{(X(s) + \varepsilon_n)^{\frac{3}{2}}} \right) ds < \infty, \quad t\in[0,T],
\end{equation}
a.s. because all other limits in \eqref{eq: Ito} as $\varepsilon_n \to 0$ are finite a.s. However, the integral $\int_0^t \frac{1}{Y(s)}ds$ which arises in \eqref{eq: place where integral of 1 over Y arises} may be infinite and thus the explicit form of the limit above for now remains obscure. This issue as well as the connection of $Y$ to the ROU process is addressed in the next theorem.

\begin{theorem}\label{th: existence of integral in equation for Y}
    Let $Y=\{Y(t), t\ge 0\} = \{\sqrt{X(t)},~t \ge 0\}$ be the square root process, where $X$ is the CIR process defined by \eqref{0.q}. Denote
    \[
        \tau := \inf \{t\ge 0: X(t)=0\} = \inf \{t\ge 0: Y(t)=0\}.
    \]
    \begin{itemize}
        \item[(a)] If $a>\frac{\sigma^2}{4}$, then for any $t\ge 0$ 
        \[
            \int_0^t \frac{1}{Y(s)}ds < \infty \quad a.s.
        \]
        Moreover, the square root process $Y$ a.s. satisfies the SDE of the form
        \begin{equation}\label{equ:posit}
            Y(t) = Y(0) + \frac{1}{2} \left(a - \frac{\sigma^2}{4}\right) \int_0^t \frac{1}{Y(s)}ds - \frac{b}{2} \int_0^t Y(s) ds + \frac{\sigma}{2} W(t), 
        \end{equation}
        $Y(0) = \sqrt{X(0)}$, and the solution to this equation is unique among non-negative stochastic processes.

        \item[(b)] If $a= \frac{\sigma^2}{4}$, then
        \[
            \int_0^{\tau}\frac{1}{Y(s)}ds < \infty \quad a.s.
        \]
        while
        \[
            \int_0^{\tau+\gamma} \frac{ds}{Y(s)}ds = \infty \quad a.s.
        \]
        for any $\gamma > 0$. Moreover, the square root process $Y$ satisfies the SDE of the form
        \begin{equation}\label{equal_1}
            Y(t) = Y(0) - \frac{b}{2}\int_0^t Y(s) ds + \frac{\sigma}{2} W(t) +  L(t),
        \end{equation}
        where the process $L$ from \eqref{equal_1} is a continuous nondecreasing process the points of growth of which can occur only at zeros of $Y$, i.e. $Y$ is a reflected Ornstein-Uhlenbeck process.
    \end{itemize}
\end{theorem}

\begin{proof}
    \emph{Case (a): $a > \frac{\sigma^2}{4}$.} Denote $p := a - \frac{\sigma^2}{4} > 0$. Our goal is to prove that the integral
    \[
        \int_0^t \frac{1}{\sqrt{X(s)}} ds = \int_0^t \frac{1}{Y(s)} ds
    \]
    is finite a.s. Define $A(t) := \left\{ \omega\in\Omega:~\int_0^t \frac{1}{\sqrt{X(s)}}ds = +\infty \right\}$ and assume that for some $t>0$:  $\mathbb P\left(A(t)\right) > 0$. Fix $T>t$, the corresponding sequence $\{\varepsilon_n,~n\ge 1\}$ such that convergence \eqref{eq: a.s.} holds and an arbitrary $\omega\in A(t) \cap \Omega'$, where $\Omega'\subset \Omega$, $\mathbb P(\Omega') = 1$, is the set where \eqref{eq: a.s.} takes place (in what follows, $\omega$ in brackets will be omitted). Then
    \begin{align*}
        \int_0^t &\left( \frac{a}{\sqrt{X(s) + \varepsilon_n}} - \frac{\sigma^2}{4} \frac{X(s)}{(X(s) + \varepsilon_n)^{\frac{3}{2}}} \right)ds 
        \\
        &= \frac{\sigma^2}{4} \int_0^t  \left( \frac{1}{\sqrt{X(s) + \varepsilon_n}} - \frac{X(s)}{(X(s) + \varepsilon_n)^{\frac{3}{2}}} \right)ds + p \int_0^t \frac{1}{\sqrt{X(s) + \varepsilon_n}} ds.
    \end{align*}
    Obviously, for all $s\in[0,t]$
    \[
        \frac{1}{\sqrt{X(s) + \varepsilon_n}} \ge \frac{X(s)}{(X(s) + \varepsilon_n)^{\frac{3}{2}}} \quad a.s.,
    \]
    so, for $\omega \in A(t) \cap \Omega'$
    \[
        \int_0^t \left( \frac{a}{\sqrt{X(s) + \varepsilon_n}} - \frac{\sigma^2}{4} \frac{X(s)}{(X(s) + \varepsilon_n)^{\frac{3}{2}}} \right)ds \to \infty \quad a.s., \quad n \to \infty,
    \]
    whence, taking into account \eqref{eq: Ito}--\eqref{eq: convergence CIR 2}, we obtain that
    \[
        \sqrt{X(t)} - \sqrt{X(0)} + \frac{b}{2} \int_0^t \sqrt{X(s)} ds - \frac{\sigma}{2} W(t) = \infty,
    \]
    which is impossible a.s. We get a contradiction, whence $\mathbb P(A(t)) = 0$ for all $t\ge 0$ and $\int_0^t \frac{1}{\sqrt{X(s)}} ds = \int_0^t \frac{1}{Y(s)} ds < \infty$ a.s. By going to the limit in \eqref{eq: Ito}, we immediately get \eqref{equ:posit}.
    
    Concerning the uniqueness of solution to \eqref{equ:posit}, let $\widetilde Y(t)$ be any of
    its non-negative  solutions. Then, by It\^o's formula,
    \[
    	\widetilde Y^2(t) = X(0) + \int_0^t \left(a - b \widetilde Y^2(s)\right) ds + \sigma \int_0^t \widetilde Y(s) dW(s)
    \]
    so $\widetilde Y$ satisfies equation $\eqref{0.q}$ and thus coincides with $X$. Therefore
    \[
        \widetilde Y(t) = \sqrt{X(t)} = Y(t) \quad a.s., \quad t\ge 0.
    \]
    
    \emph{Case (b): $a = \frac{\sigma^2}{4}$.} Fix $T>0$ and take the corresponding sequence $\{\varepsilon_n,~n\ge 1\}$ such that \eqref{eq: a.s.} holds. By \eqref{eq: Ito}, for any $t\in[0,T]$
    \begin{align*}
        \sqrt{X(t) + \varepsilon_n} &= \sqrt{X(0) + \varepsilon_n} + \frac{1}{2} \int_0^t \left( \frac{a}{\sqrt{X(s) + \varepsilon_n}} - \frac{\sigma^2}{4} \frac{X(s)}{(X(s) + \varepsilon_n)^{\frac{3}{2}}} \right)ds 
        \\
        &\quad - \frac{1}{2} \int_0^t \frac{b X(s)}{\sqrt{X(s) + \varepsilon_n}} ds + \frac{\sigma}{2} \int_0^t \frac{\sqrt{X(s)}}{\sqrt{X(s) + \varepsilon_n}} dW(s)
        \\
        &= \sqrt{X(0) + \varepsilon_n} + \frac{\sigma^2}{8} \int_0^t  \frac{ \varepsilon_n }{(X(s) + \varepsilon_n)^{\frac{3}{2}}} ds 
        \\
        &\quad - \frac{1}{2} \int_0^t \frac{b X(s)}{\sqrt{X(s) + \varepsilon_n}} ds + \frac{\sigma}{2} \int_0^t \frac{\sqrt{X(s)}}{\sqrt{X(s) + \varepsilon_n}} dW(s),
    \end{align*}
    and \eqref{eq: the reflected piece a.s.} implies that there exists $\Omega'\subset\Omega$, $\mathbb P(\Omega')=1$, such that for all $\omega\in\Omega'$ the limit
    \[
        L(t) := \lim_{n\to \infty} \frac{\sigma^2}{8} \int_0^t \frac{\varepsilon_n}{\left(X(s) + \varepsilon_n\right)^\frac{3}{2}} ds
    \]
    is well-defined and finite for all $t\in[0,T]$. It is evident that $L(0) = 0$ a.s. due to continuity of $X$ and the fact that $X(0)>0$. Moreover, since a.s. 
    \begin{equation*}
        L(t) = \sqrt{X(t)} - \sqrt{X(0)} + \frac{b}{2} \int_0^t \sqrt{X(s)} ds - \frac{\sigma}{2} W(t) \quad t\in[0,T], 
    \end{equation*}
    $L$ is continuous in $t$. Furthermore,
    \[
        \int_0^{t_1} \frac{\varepsilon_n}{\left(X(s) + \varepsilon_n\right)^\frac{3}{2}} ds \le \int_0^{t_2} \frac{\varepsilon_n}{\left(X(s) + \varepsilon_n\right)^\frac{3}{2}} ds
    \]
    for all $t_1 < t_2$, and whence $L$ is non-decreasing in $t$ a.s. Finally, if $X(t) = x > 0$, there exists an interval $[t_1, t_2]$ containing $t$ such that $X(s) > \frac{x}{2}$ for all $s\in [t_1, t_2]$ and thus
    \[
        L(t_1) - L(t_2) = \lim_{n\to \infty} \frac{\sigma^2}{8} \int_{t_1}^{t_2} \frac{\varepsilon_n}{\left(X(s) + \varepsilon_n\right)^\frac{3}{2}} ds \to 0, \quad n\to\infty,
    \]
    i.e. $L$ can increase only at points of zero hitting of $X$ that coincide with the ones of $Y$. Taking into the account all of the above as well as an arbitrary choice of $T$, $L$ is the reflection function for $Y$ and the latter is indeed a ROU process.
    
    Now, let us prove that $\int_0^\tau \frac{1}{Y_s}ds < \infty$ a.s. Consider a standard Ornstein-Uhlenbeck process $U = \{U(t),~t \ge 0\}$ of the form
    \begin{equation}\label{eq: OU process U}
        U(t) = \sqrt{X(0)} - \frac{b}{2} \int_0^t U(s) ds + \frac{\sigma}{2} W(t), 
    \end{equation}
    with $W$ being the same Brownian motion that drives $X$. It is evident that $Y$ coincides with $U$ until $\tau$ a.s. and thus it is sufficient to prove that $\int_0^\tau \frac{1}{U(s)}ds < \infty$ a.s. For any $\varepsilon > 0$ consider
    \[
        {\frac{\sigma^2}{4}}\int_0^\tau \frac{1}{U(s)} \mathbbm 1_{\{\varepsilon < U(s) < 1\}} ds = \int_{\varepsilon}^{1} \frac{L_U(\tau,x)}{x} dx,
    \]
    where $L_U$ denotes the local time of $U$, and observe that
    \[
        {\frac{\sigma^2}{4}} \int_0^\tau \frac{1}{U(s)}ds \le \lim_{\varepsilon\downarrow 0} \int_0^\tau \frac{1}{U(s)} \mathbbm 1_{\{\varepsilon < U(s) < 1\}} ds = \int_0^1 \frac{L_U(\tau,x)}{x}dx.
    \]
    Computations similar to the ones in \cite[Section IV.44]{Rogers_Williams_2000} indicate that the local time $L_U(t,x)$ of $U$ is H\"older continuous in $x$ up to order $\frac{1}{2}$ over bounded time intervals and thus $\int_0^\tau \frac{1}{U(s)}ds = \int_0^\tau \frac{1}{Y(s)}ds < \infty$ a.s.

	Finally, assume that  for some $\gamma > 0$,
	\[
	    \int_0^{\tau+\gamma} \frac{1}{Y(s)}ds < \infty
	\]
	with positive probability. On $\omega \in \Omega$ where this property holds, we have that
	\begin{align*}
		\int_0^{\tau + \gamma}& \frac{ds}{\sqrt{X(s) + \varepsilon}} - \int_0^{\tau + \gamma} \frac{X(s)}{(X(s) + \varepsilon)^{\frac{3}{2}}}ds 
		\\
		&\to \int_0^{\tau+\gamma} \frac{1}{Y(s)}ds - \int_0^{\tau+\gamma} \frac{1}{Y(s)}ds = 0, \quad \varepsilon \to 0.
	\end{align*}
	Therefore, for such $\omega$, $Y$ satisfies the equation of the form
	\[
		Y(t) = Y(0) - \frac{b}{2} \int_0^t Y(s) ds + \frac{\sigma}{2} W(t)
	\]
	on the interval $[0,\tau + \gamma]$, i.e. such paths of $Y$ coincide with the corresponding paths of the Ornstein-Uhlenbeck process $U$ defined by \eqref{eq: OU process U} up until $\tau + \gamma$. This implies that $U(\tau) = 0$ and $U$ is nonnegative on the interval $[\tau, \tau+\gamma]$ for such $\omega$, which is impossible due to the non-tangent property of Gaussian processes stated by \cite{Ylvisaker1968}, see also \cite{Piterbarg2015}.
\end{proof}

\begin{remark}\label{rem: epsilonn can be repaced by epsilon}
    Since the integral $\int_0^t \frac{1}{\sqrt{X(s)}} ds$ is finite a.s. for $a > \frac{\sigma^2}{4}$,
    \begin{gather*}
        \sqrt{X(t) + \varepsilon} - \sqrt{X(0) + \varepsilon} - \frac{1}{2} \int_0^t \left( \frac{a}{\sqrt{X(s) + \varepsilon}} - \frac{\sigma^2}{4} \frac{X(s)}{(X(s) + \varepsilon)^{\frac{3}{2}}} \right)ds 
        \\
        + \frac{1}{2} \int_0^t \frac{b X(s)}{\sqrt{X(s) + \varepsilon}} ds
        \\
        \xrightarrow{a.s.} \sqrt{X(t)} - \sqrt{X(0)} - \frac{1}{2}\left(a - \frac{\sigma^2}{4}\right) \int_0^t \frac{1}{\sqrt{X(s)}} ds + \frac{b}{2} \int_0^t \sqrt{X(s)} ds < \infty
    \end{gather*}
    as $\varepsilon \to 0$. Therefore, taking into account \eqref{eq: Ito},
    \[
        \int_0^t \frac{\sqrt{X(s)}}{\sqrt{X(s) + \varepsilon}} dW(s) \xrightarrow{a.s.} W(t), \quad \varepsilon \to 0.
    \]
\end{remark}

As a corollary of Theorem \ref{th: existence of integral in equation for Y}, we have a representation of the reflection function of the ROU process as the limit of integral functionals of the CIR processes. It is interesting that the reflection function is singular w.r.t. the Lebesgue measure (see Remark \ref{rem: reflection function is local time}) while the processes that converge to it are absolutely continuous a.s.

\begin{theorem}\label{th: representation of ROU process}
    Let $\{W(t),~t\ge 0\}\}$ be a continuous modification of a standard Brownian motion, $Y(0)$, $b$, $\sigma > 0$ be given constants and  $\{\varepsilon_n,~n\ge 1\}$ be an arbitrary sequence such that $\varepsilon_n \downarrow 0$, $n\to\infty$. For any $\varepsilon_n$ from this sequence, consider the CIR process $X_{\varepsilon_n} = \{X_{\varepsilon_n}(t),~t\ge 0\}$ given by
    \[
        X_{\varepsilon_n}(t) = X(0) + \int_0^t \left(\frac{\sigma^2}{4} + \varepsilon_n - bX_{\varepsilon_n}(s)  \right) ds + \sigma\int_0^t \sqrt{X_{\varepsilon_n}(s)}dW(s)
    \]
    and denote its square root by $Y_{\varepsilon_n}(t) := \sqrt{X_{\varepsilon_n}(t)}$.
    Then, with probability 1,
    \begin{itemize}
        \item[1)] the limit $\lim_{n\to\infty } Y_{\varepsilon_n}(t) =: Y(t)$ is well-defined, finite and non-negative for any $t\ge 0$;
        \item[2)] the limit process $Y = \{Y(t),~t\ge 0\}$ is a ROU process satisfying the equation of the form
        \begin{equation*}
            Y(t) = Y(0) - \frac{b}{2} \int_0^t Y(s) ds + \frac{\sigma}{2} W(t) + L(t),\quad t\ge 0,
        \end{equation*}
        with $Y(0) = \sqrt{X(0)} >0$ and $L$ being the reflection function for $Y$;
        \item[3)] for any $T>0$
        \begin{equation}\label{eq: uniform convergence 1, standard case}
            \sup_{t\in[0,T]} |Y(t) - Y_{\varepsilon_n}(t)| \to 0, \quad n\to\infty,
        \end{equation}
        and
        \begin{equation}\label{eq: uniform convergence 2, standard case}
            \sup_{t\in[0,T]} \left|L(t) - \frac{1}{2}\int_0^t \frac{\varepsilon_n}{Y_{\varepsilon_n}(s)}ds\right| \to 0, \quad n\to\infty.
        \end{equation}
    \end{itemize}
\end{theorem}
\begin{proof}
    Denote $X$ the CIR process of the form
    \[
        X(t) = X(0) + \int_0^t \left(\frac{\sigma^2}{4} - bX(s)  \right) ds + \sigma\int_0^t \sqrt{X(s)}dW(s).
    \]
    By Theorem \ref{th: existence of integral in equation for Y}, there exists $\Omega'\subset\Omega$, $\mathbb P(\Omega')=1$, such that for all $\omega \in \Omega'$ $X$ and each $Y_{\varepsilon_n}$, $n\ge 1$, are continuous and the latter satisfy equations of the form
    \[
        Y_{\varepsilon_n}(t) = Y(0) + \frac{1}{2}\int_0^t \frac{\varepsilon_n}{Y_{\varepsilon_n}(s)}ds - \frac{b}{2} \int_0^t Y(s)ds + \frac{\sigma}{2}W(t), \quad t\ge 0,
    \]
    with the integral $\int_0^t \frac{1}{Y_{\varepsilon_n}(s)}ds < \infty$. Furthermore, since each $X_{\varepsilon_n} = Y^2_{\varepsilon_n}$ is a CIR process that satisfies conditions of the comparison theorem from \cite{IW1977}, this $\Omega'$ can be chosen such that for all $\omega\in\Omega'$
    \begin{equation}\label{eq: decreasing Y and X}
        Y_{\varepsilon_{n}}(\omega, t) \ge Y_{\varepsilon_{n+1}}(\omega, t) \ge \sqrt{X(t)} \ge 0,\quad t\ge 0, \quad n\ge 1.
    \end{equation}
    
    Fix $\omega\in\Omega'$ (in what follows, we will omit $\omega$ in brackets for notational simplicity). Since the sequence $\{Y_{\varepsilon_n}(t),~n\ge 1\}$ is non-increasing for each $t\ge 0$, there exists a pointwise limit $Y(t) := \lim_{n\to\infty} Y_{\varepsilon_n}(t) \in [0,\infty)$. Moreover, it is evident that $\lim_{n\to\infty} \int_0^t Y_{\varepsilon_n}(s) ds = \int_0^t Y(s) ds$ and since
    \begin{align*}
        Y(t) &= \lim_{n\to\infty} Y_{\varepsilon_n}(t) 
        \\
        &= Y(0) - \lim_{n\to\infty}\frac{b}{2}\int_0^t Y_{\varepsilon_n}(s)ds + \frac{\sigma}{2} W(t) +  \lim_{n\to\infty}\frac{1}{2} \int_0^t \frac{\varepsilon_n}{Y_{\varepsilon_n}(s)}ds 
        \\
        &= Y(0)  - \frac{b}{2}\int_0^t Y(s) ds + \frac{\sigma}{2} W(t) + \lim_{n\to\infty}\frac{1}{2}\int_0^t \frac{\varepsilon_n}{Y_{\varepsilon_n}(s)}ds,
    \end{align*}
    the limit $L(t) :=  \lim_{n\to\infty}\frac{1}{2}\int_0^t \frac{\varepsilon_n}{Y_{\varepsilon_n}(s)}ds$ is well-defined, nonnegatove and finite. 
    
    In order to obtain the claim of the theorem, it is sufficient to check that the function $L$ defined above is indeed a reflection function for $Y$, i.e. is continuous and nondecreasing process that starts at zero and the points of growth of which occur only at zeros of $Y$. Note that continuity of $L$ would also imply the uniform convergences \eqref{eq: uniform convergence 1, standard case} and \eqref{eq: uniform convergence 2, standard case} on each compact $[0,T]$. Indeed, since $Y_{\varepsilon_n}(t) \ge Y_{\varepsilon_{n+1}}(t)$ for all $t\ge 0$, $n\ge 1$ and continuity of $L$ would imply continuity of $Y$, Dini's theorem guarantees \eqref{eq: uniform convergence 1, standard case}. The same argument applies to \eqref{eq: uniform convergence 2, standard case}: the right-hand side of 
    \[
        \frac{1}{2} \int_0^t \frac{\varepsilon_n}{Y_{\varepsilon_n}(s)}ds = Y_{\varepsilon_n}(t) - Y(0) + \frac{b}{2} \int_0^t Y_{\varepsilon_n}(s)ds - \frac{\sigma}{2}W(t)
    \]
    is non-increasing w.r.t. $t$, therefore for each $t\ge 0$ and $n\ge 1$
    \[
        \frac{1}{2} \int_0^t \frac{\varepsilon_n}{Y_{\varepsilon_n}(s)}ds \ge \frac{1}{2} \int_0^t \frac{\varepsilon_{n+1}}{Y_{\varepsilon_{n+1}}(s)}ds
    \]
    and Dini's theorem implies \eqref{eq: uniform convergence 2, standard case} as well.
    
    By \eqref{eq: decreasing Y and X}, continuity of $X$ and the fact that $X(0) >0$, there exists an interval $[0,t_0)$ such that for all $t\in[0,t_0)$ and $n\ge 1$ $Y_{\varepsilon_n}(t) \ge \frac{Y(0)}{2}$. Thus for any $t\in [0,t_0)$
    \[
        L(t) = \lim_{n\to\infty} \frac{1}{2}\int_0^t \frac{\varepsilon_n}{Y_{\varepsilon_n}(s)}ds \le \lim_{n\to\infty}\frac{t_0\varepsilon_n}{Y(0)} = 0,
    \]
    i.e. $L(t) = 0$ for all $t\in[0,t_0]$.
    
    For reader's convenience, we will split the further proof into four steps.
    
    \textbf{Step 1: $L$ in non-decreasing.} Monotonicity of $L$ is obvious since for any fixed $n\ge1$ and $t_1 < t_2$
    \[
        \int_0^{t_1} \frac{\varepsilon_n}{Y_{\varepsilon_n}(s)}ds \le \int_0^{t_2} \frac{\varepsilon_n}{Y_{\varepsilon_n}(s)}ds.
    \]
    
    \textbf{Step 2: right-continuity.} Let us show that $L$ is continuous from the right. For any fixed $t\ge 0$, denote $L({t+}) := \lim_{\delta \downarrow 0} L({t+\delta})$ (the right limit exists since $L$ is non-decreasing) and assume that $L({t+}) - L({t}) = \alpha > 0$. Due to the monotonicity of $L$, this implies that for all $\delta > 0$
    \begin{equation}\label{proofeq: L is not right continuous}
        L({t+\delta}) - L({t}) \ge \alpha > 0.
    \end{equation}
    Now, take $n_0$ such that for all $n\ge n_0$
    \[
        \frac{1}{2} \int_0^t \frac{\varepsilon_n}{Y_{\varepsilon_n} (s)}ds \in \left[L(t), L(t) + \frac{\alpha}{4}\right)
    \]
    and $\delta_0 >0$ such that  
    \[
        \frac{1}{2} \int_0^{t+\delta_0} \frac{\varepsilon_{n_0}}{Y_{\varepsilon_{n_0}}(s)}ds \in \left[L(t), L(t) + \frac{\alpha}{2}\right).
    \]
    As it was noted previously, for each $s\ge 0$ the values of $\frac{1}{2}\int_0^s \frac{\varepsilon_n}{Y_{\varepsilon_n}(u)}du$ are non-increasing when $n\to\infty$. Thus for any $n \ge n_0$
    \[
        \frac{1}{2} \int_0^{t+\delta_0} \frac{\varepsilon_n}{Y_{\varepsilon_n}(s)}ds \le \frac{1}{2} \int_0^{t+\delta_0} \frac{\varepsilon_{n_0}}{Y_{\varepsilon_{n_0}}(s)}ds \le L(t) + \frac{\alpha}{2},
    \]
    i.e.
    \[
        L({t+}) \le \lim_{n\to\infty} \frac{1}{2} \int_0^{t+\delta_0} \frac{\varepsilon_n}{Y_{\varepsilon_n}(s)}ds < L(t) + \frac{\alpha}{2},
    \]
    which contradicts \eqref{proofeq: L is not right continuous}. Therefore, $L({t+}) - L(t) = 0$, i.e. $L$ is right-continuous. 
    
    \textbf{Step 3: left-continuity.} Now, let us show that $L$ is continuous from the left. Assume that it is not true and there exists $t > 0$ such that $L({t}) - L({t-}) > 0$ (note that $L({t-}) = \lim_{\delta\downarrow 0} L({t-\delta})$ is well-defined due to the monotonicity of $L$). Since $L$ may have only positive jumps, so does $Y$ and, moreover, the points of jumps of $L$ and $Y$ coincide. This implies that $Y({t}) - Y({t-}) > 0$ and we now consider two cases.
    
    \emph{Case 1: $Y({t-}) = y > 0$.} Then $Y(t) = Y({t+}) > y$ (note that $Y$ is right-continuous by Step 2) and there exists an interval $[t-\delta, t+\delta]$ such that $Y_{\varepsilon_n} (s) \ge Y(s) > \frac{y}{2}$ for all $s \in [t-\delta, t+\delta]$. This implies that
    \[
        L({t+\delta}) - L({t-\delta}) = \lim_{n\to\infty} \frac{1}{2}\int_{t-\delta}^{t+\delta} \frac{\varepsilon_n}{Y_{\varepsilon_n}(s)}ds \le \lim_{n\to\infty} \frac{2\delta \varepsilon_n}{y} = 0,
    \]
    i.e. $L$ cannot have a jump at $t$. This means that $Y$ cannot have a jump at point $t$ either and we obtain a contradiction.
    
    \emph{Case 2: $Y({t-}) = 0$ and $Y({t+}) = Y({t}) = y > 0$.} Fix $T>t$, $\lambda \in \left(0,\frac{1}{2}\right)$ and let $\Lambda$ be a random variable such that for all $t_1, t_2 \in [0,T]$
    \[
        |W({t_1}) - W({t_2})| \le \Lambda |t_1 - t_2|^\lambda.
    \]
    Take $n_1 \ge 1$ and $\delta_1 > 0$ such that $\varepsilon_{n_1} + \delta_1 + \frac{\sigma\Lambda}{2} \delta_1^\lambda < y$ and note that there exists $\delta_2 < \delta_1$ such that $Y({t-\delta_2}) < \varepsilon_{n_1}$. Since $Y_{\varepsilon_n}(t-\delta_2) \downarrow Y({t-\delta_2})$ as $n\to\infty$, there exists $n_2 > n_1$ such that $Y_{\varepsilon_{n_2}}(t-\delta_2) < \varepsilon_{n_1} < y$. Moreover, $Y_{\varepsilon_{n_2}}(t) \ge Y(t) = y$ thus one can define
    \[
        \tau := \sup\{s\in (t-\delta_2, t),~Y_{\varepsilon_{n_2}}(s) =\varepsilon_{n_1}\}.
    \]
    Observe that $Y_{\varepsilon_{n_2}} (\tau) = \varepsilon_{n_1}$ and $Y_{\varepsilon_{n_2}}(s) \ge \varepsilon_{n_1}$ for all $s\in[\tau, t]$, whence
    \begin{align*}
        Y_{\varepsilon_{n_2}}(t) &= Y_{\varepsilon_{n_2}}(\tau) + \frac{1}{2}\int_\tau^t \frac{\varepsilon_{n_2}}{Y_{\varepsilon_{n_2}}(s)}ds - \frac{b}{2}\int_\tau^t Y_{\varepsilon_{n_2}}(s)ds + \frac{\sigma}{2} (W(t) - W(\tau))
        \\
        &\le \varepsilon_{n_1} + \frac{\varepsilon_{n_2}}{2\varepsilon_{n_1}} (t-\tau) + \frac{\sigma \Lambda}{2}(t-\tau)^\lambda
        \\
        &\le \varepsilon_{n_1} + \delta_1 + \frac{\sigma\Lambda}{2} \delta_1^\lambda < y,
    \end{align*}
    which contradicts the assumption that $Y_{\varepsilon_{n_2}}(t) \ge y$. This contradiction together with all of the above implies that $Y$ (and thus $L$) is continuous at each point $t \ge 0$.
    
    \textbf{Step 4: points of growth.} Now, let us prove that the points of growth of $L$ may occur only at zeros of $Y$. Indeed, let $t > 0$ be such that $Y(t) = y > 0$. Since $Y$ is continuous, there exists $\delta_3 > 0$ such that for any $s \in(t-\delta_3,t+\delta_3)$
    \[
        Y({s}) > \frac{y}{2} > 0.
    \]
    This, in turn, implies that for all $s \in(t-\delta_3,t+\delta_3)$ and $n\ge 1$
    \[
        Y_{\varepsilon_n}(s) \ge Y({s}) > \frac{y}{2} > 0
    \]
    and thus for any $\delta \in [0,\delta_3)$
    \[
        L({t+\delta}) - L({t-\delta}) = \lim_{n\to\infty} \frac{1}{2} \int_{t-\delta}^{t+\delta} \frac{\varepsilon_n}{Y_{\varepsilon_n}(s)}ds \le \lim_{n\to\infty} \frac{2\delta }{y} \varepsilon_n = 0.
    \]
    Therefore $L({t+\delta}) - L({t-\delta})=0$ and $L$ does not grow in some neighbourhood of $t$.
\end{proof}
	
\begin{remark}\label{rem: reflection function is local time}
    It is well-known (see e.g. \cite[Appendix A]{BallRoma1994} or \cite[Subsection 3.3.1]{Zhu2010}) that the absolute value of OU and ROU processes with non-zero mean reversion levels do not coincide. In turn, in the ``symmetric'' case with zero mean reversion parameter, absolute value of the OU process and ROU process have the same distribution but do not coincide pathwisely. Theorem \ref{th: existence of integral in equation for Y} allows to clarify this subtle difference in the following manner.
    
    Let $B = \{B(t),~t\ge 0\}$ be some standard Brownian motion and
    \[
        U(t) = U(0) - \frac{b}{2} \int_0^t U(s) ds + \frac{\sigma}{2} B(t), \quad t\ge 0,
    \]
    be a standard Ornstein-Uhlenbeck process with non-random positive initial value $U(0) > 0$. By It\^o's formula,
    \begin{align*}
        U^2(t) &= U^2(0) + \int_0^t \left( \frac{\sigma^2}{4} - b U^2(s)\right)ds + \sigma \int_0^t U(s) dB(s)
        \\
        & = U^2(0) + \int_0^t \left( \frac{\sigma^2}{4} - b U^2(s)\right)ds + \sigma \int_0^t |U(s)| \sign(U(s)) dB(s)
        \\
        &= U^2(0) + \int_0^t \left( \frac{\sigma^2}{4} - b U^2(s)\right)ds + \sigma \int_0^t |U(s)| dW(s),
    \end{align*}
    where $W(t) := \int_0^t \sign(U(s))dB(s)$ is a standard Brownian motion (which can be easily verified by Levy's characterization). Thus, the process $X(t) := U^2(t)$, $t\ge 0$, is a CIR process w.r.t. $W$. By Theorem \ref{th: existence of integral in equation for Y}, the square root process $Y(t) := \sqrt{X(t)}$, $t\ge 0$, is a reflected Ornstein-Uhlenbeck process with respect to $W$ satisfying the SDE of the form
    \begin{equation}\label{eq: reflection function as local time 1}
        Y(t) = U(0) - \frac{b}{2} \int_0^t Y(s) ds + \frac{\sigma}{2} W(t) + L(t), \quad t\ge 0,
    \end{equation}
    with $L$ being the reflection function for $Y$. Since $Y(t) =\sqrt{X(t)} = |U(t)|$, by Tanaka's formula
    \begin{equation}\label{eq: reflection function as local time 2}
    \begin{aligned}
        Y(t) &= U(0) + \int_0^t \sign(U(s)) dU(s) + L_U(t)
        \\
        &= U(0) - \frac{b}{2}\int_0^t \sign(U(s)) U(s) ds + \frac{\sigma}{2}\int_0^t \sign(U(s)) dB(s) + L_U(t)
        \\
        &= U(0) - \frac{b}{2}\int_0^t Y(s) ds + \frac{\sigma}{2} W(t) + L_U(t),
    \end{aligned}
    \end{equation}
    with $L_U$ being the local time of $U$ at zero. Comparing \eqref{eq: reflection function as local time 1} and \eqref{eq: reflection function as local time 2}, we obtain that $L(t) = L_U(t)$, i.e. the reflection function of the ROU process $Y$ coincides with local time at zero of the OU process $U$.
\end{remark}

\section{Fractional Cox-Ingersoll-Ross and fractional reflected Ornstein-Uhlenbeck processes}\label{sec: fractional case}
	
Let now $\{B^H(t),~t\ge 0\}$ be a continuous modification of a fractional Brownian motion with Hurst index $H>\frac{1}{2}$. Consider a stochastic differential equation of the form
\begin{equation}\label{eq: FCIR}
    Y^H(t) = Y(0) + \frac{1}{2} \int_0^t \left( \frac{a}{Y^H(s)} - b Y^H(s) \right)ds + \frac{\sigma}{2} dB^H(t), t\ge 0,
\end{equation}
where $Y(0) > 0$ is a given constant, $a$, $b$, $\sigma > 0$. According to \cite{MYuT2018} (see also \cite{DNMYT2020}), SDE \eqref{eq: FCIR} a.s. has a unique pathwise solution $\{Y^H(t),~t\ge 0\}$ such that $Y^H(t) > 0$ for all $t\ge 0$, and the subset of $\Omega$ where this solution exists does not depend on $Y(0)$, $a$, $b$ or $\sigma$ (in fact, the solution exists for all $\omega\in\Omega$ such that $B^H(\omega, t)$ is locally H\"older continuous in $t$). Moreover, it can be shown (see \cite[Theorem 1]{MYuT2018} or \cite[Subsection 4.1]{DNMYT2020}) that the process $X^H(t) = (Y^H(t))^2$, $t\ge 0$, satisfies the SDE of the form
\begin{equation}\label{eq: FCIR true form}
    X^H(t) = X(0) + \int_0^t (a-bX^H(s))ds + \sigma \int_0^t \sqrt{X^H(s)} dB^H(s), \quad t\ge 0,
\end{equation}
where $X(0) = Y^2(0)$ and the integral with respect to the fractional Brownian motion exists as the pathwise limit of the corresponding Riemann-Stieltjes integral sums. Taking into account the form of \eqref{eq: FCIR true form}, the process $\{X^H(t),~t\ge 0\}$ can be interpreted as a natural fractional generalisation of the Cox-Ingersoll-Ross process with $\{Y^H(t),~t\ge 0\}$ being its square root.

\begin{remark}
    It is evident that the solution to \eqref{eq: FCIR true form} is unique in the class of non-negative stochastic processes with paths that are H\"older-continuous up to the order $H$. Indeed, by the fractional pathwise counterpart of the It\^o's formula (see e.g. \cite[Theorem 4.3.1]{Zahle1998}) the square root of the solution must satisfy the equation \eqref{eq: FCIR} until the first moment of zero hitting. However, as it was noted above, the solution to \eqref{eq: FCIR} is unique and strictly positive a.s., i.e. never hits zero.
\end{remark}

Now, let us recall the definition of the reflected fractional Ornstein-Uhlenbeck (RFOU) process.
\begin{definition}
    Stochastic process $\widetilde Y^H = \{\widetilde Y^H(t),~t\ge 0\}$ is called a \emph{fractional reflected Ornshein-Uhlenbeck (RFOU) process} if it satisfies a stochastic differential equation of the form
    \begin{equation}\label{eq: fractional reflected OU process equation}
        \widetilde Y^H(t) = Y(0) - \widetilde{b} \int_0^t \widetilde Y^H(s) ds + \widetilde{\sigma} B^H(t) + \widetilde L^H(t),\quad t\ge 0,
    \end{equation}
    where $Y(0)$, $\widetilde{b}$ and $\widetilde{\sigma}$ are positive constants, $B^H = \{B^H(t),~t\ge 0\}$ is a fractional Brownian motion, $\{\widetilde L^H(t),~t \ge 0\}$ is a reflection function for $\widetilde Y^H$ in the sense of Definition \ref{def: reflection function} and $\widetilde Y^H \ge 0$ a.s.
\end{definition}

\begin{remark}
    For more details on properties of the RFOU process see e.g. \cite{LeeSong2016} and references therein. Note that, by the argument similar to the one stated in Remark \ref{rem: uniqueness of reflection function}, the solution $(Y^H, L^H)$ to the equation \eqref{eq: fractional reflected OU process equation} is unique.
\end{remark}

When it comes to the connection between FCIR and RFOU processes, there is a notable difference from the standard Brownian case discussed in section \ref{sec: standard case}: in the standard case the ROU process turned out to coincide with the square root of the CIR process with $a = \frac{\sigma^2}{4}$ which is not true for the fractional case. More precisely, if $a>0$, $X^H$ is strictly positive a.s. and thus $\sqrt{X^H}$ cannot coincide with the RFOU process. Furthermore, for $a=0$ \cite[Theorem 6]{MMS2015} claims existence and uniqueness of solution to \eqref{eq: FCIR true form} when $H\in\left(\frac{2}{3}, 1\right)$, and this solution turns out to stay in zero after hitting it, i.e. its square root is also different from the RFOU process. However, it is still possible to establish a clear connection between FCIR and RFOU processes highlighted in the next theorem.

\begin{theorem}\label{th: representation of FROU process}
    Let $\{B^H(t),~t\ge 0\}$ be a continuous modification of a fractional Brownian motion with Hurst index $H>\frac{1}{2}$, $Y(0)$, $b$, $\sigma > 0$ be given constants. For any $\varepsilon > 0$, consider a square root process $Y^H_\varepsilon = \{Y^H_{\varepsilon}(t),~t\ge 0\}$ given by
    \[
        Y^H_\varepsilon (t) = Y(0) + \frac{1}{2} \int_0^t\left( \frac{\varepsilon}{Y^H_\varepsilon(s)} - b Y^H_\varepsilon(s) \right)ds + \frac{\sigma}{2} B^H(t), \quad t\ge 0.
    \]
    Then, with probability 1,
    \begin{itemize}
        \item[1)] the limit $\lim_{\varepsilon \downarrow 0} Y^H_\varepsilon(t) =: Y^H(t)$ is well-defined, finite and non-negative for any $t\ge 0$;
        \item[2)] the limit process $Y^H = \{Y^H(t),~t\ge 0\}$ is a RFOU process satisfying the equation of the form
        \begin{equation}\label{eq: reflected limit of FCIR}
            Y^H(t) = Y(0) - \frac{b}{2} \int_0^t Y^H(s) ds + \frac{\sigma}{2} B^H(t) + L^H(t),\quad t\ge 0;
        \end{equation}
        \item[3)] for any $T>0$
        \[
            \sup_{t\in[0,T]} |Y^H(t) - Y^H_\varepsilon(t)| \to 0, \quad \varepsilon \downarrow 0,
        \]
        and
        \[
            \sup_{t\in[0,T]} \left|L^H(t) - \frac{1}{2}\int_0^t \frac{\varepsilon}{Y^H_\varepsilon(s)}ds\right| \to 0, \quad \varepsilon \downarrow 0.
        \]
    \end{itemize}
\end{theorem}

\begin{proof}
    Let $\omega \in \Omega$ such that $B^H(\omega, t)$ is locally H\"older continuous in $t$ be fixed (for notational simplicity, we will omit it in brackets). As it was noted above, for such $\omega$ all $Y^H_\varepsilon$ are well-defined and strictly positive. Moreover, by the comparison theorem (see e.g. \cite[Lemma 1]{MYuT2018} or \cite[Lemma A.1]{DNMYT2020}),
    for all $t\ge 0$ and $\varepsilon_1 > \varepsilon_2$ 
    \[
        Y^H_{\varepsilon_1}(t) > Y^H_{\varepsilon_2}(t) > 0. 
    \]
    This implies that for any fixed $t\ge 0$ the limits $\lim_{\varepsilon \downarrow 0} Y_\varepsilon^H(t) = Y^H(t)$ and $\lim_{\varepsilon \downarrow 0} \frac{1}{2}\int_0^t \frac{\varepsilon}{Y^H_\varepsilon(s)}ds = : L^H(t)$ are well-defined, non-negative and finite. Furthermore, by comparison theorem, each $Y^H_\varepsilon$ exceeds the fractional Ornstein-Uhlenbeck of the form
    \[
        U^H(t) = Y(0) - \frac{b}{2}\int_0^t U^H(s)ds + \frac{\sigma}{2} B^H(t), \quad t\ge 0,
    \]
    and hence there exists an interval $[0,t_0)$ such that for all $t\in[0,t_0)$ and $\varepsilon\ge 0$ $Y^H_\varepsilon(t) \ge \frac{Y(0)}{2}$. Thus for any $t\in [0,t_0)$
    \[
        L^H(t) = \lim_{\varepsilon \downarrow 0} \frac{1}{2}\int_0^t \frac{\varepsilon}{Y^H_\varepsilon(s)}ds \le \lim_{\varepsilon \downarrow 0}\frac{t_0\varepsilon}{Y(0)} = 0,
    \]
    i.e. $L^H(t) = 0$ for all $t\in[0,t_0)$.
    
    The remaining part of the proof is identical to the one of Theorem \ref{th: representation of ROU process}.
\end{proof}

\begin{remark}
    Theorem \ref{th: representation of FROU process} and the preceding remark highlight that the FCIR process \eqref{eq: FCIR true form} is not continuous at zero w.r.t. the mean-reversion parameter $a$.
\end{remark}

\section{Simulations}\label{sec: simulations}

Let us illustrate the results with simulations. On Fig. \ref{Fig1}, the black paths depict simulated trajectories of the square root $\{Y^H_\varepsilon(t),~t\ge 0\}$ of the FCIR process given by an equation of the form
\[
    Y^H_\varepsilon(t) = Y(0) + \frac{1}{2}\int_0^t \frac{\varepsilon}{Y^H_{\varepsilon}(s)}ds - \frac{b}{2}\int_0^t Y^H_{\varepsilon}(s)ds + \frac{\sigma}{2}B^H(t)
\]
with $Y(0) = 0.25$, $b=1$, $\sigma = 1$, $\varepsilon = 0.0001$ and different Hurst indices $H$; the red lines are the corresponding integrals $\frac{1}{2}\int_0^t \frac{\varepsilon}{Y^H_\varepsilon(s)}ds$. In order to simulate $Y^H_\varepsilon$, the backward Euler approximation technique from \cite{Kubilius2020} was used, see also \cite{HHKW2020, ZhYu2020}. 

\begin{figure}[h!]
  \centering
\begin{minipage}[b]{0.48\textwidth} \centering
    \includegraphics[width=\textwidth]{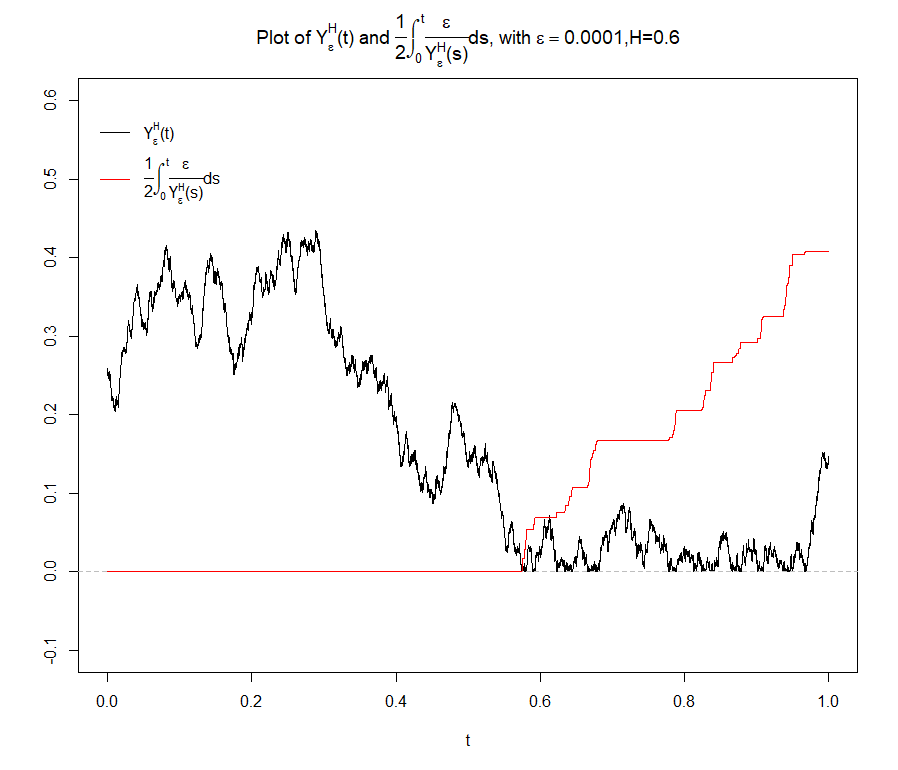}
(a) $H=0.6$
\end{minipage}
\begin{minipage}[b]{0.48\textwidth} \centering
\includegraphics[width=\textwidth]{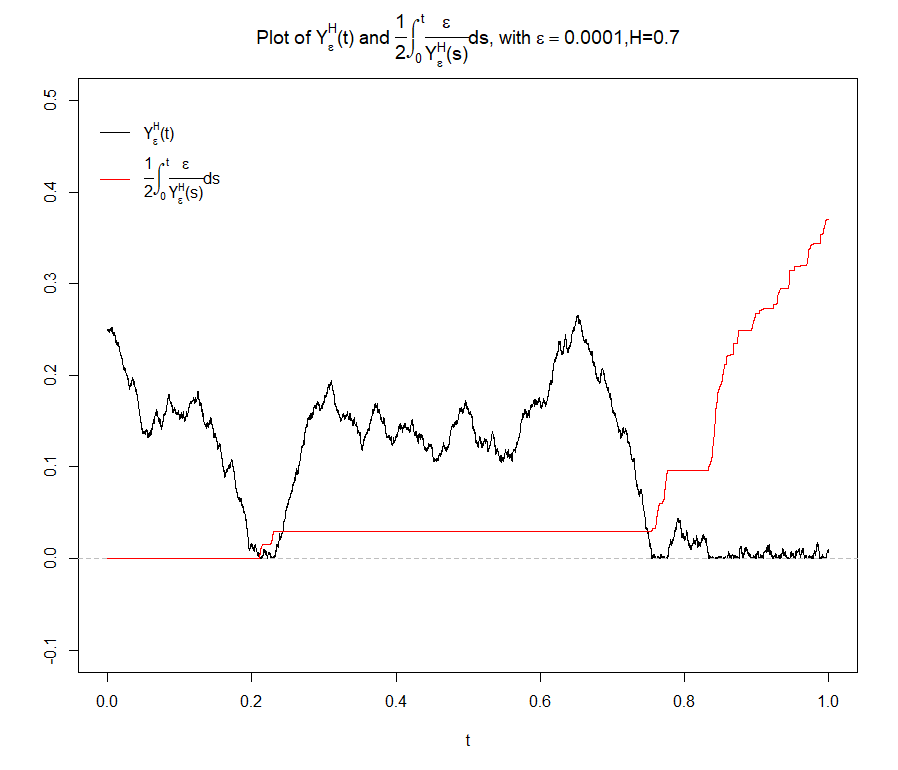}
(b) $H=0.7$
\end{minipage}
\begin{minipage}[b]{0.48\textwidth} \centering
\includegraphics[width=\textwidth]{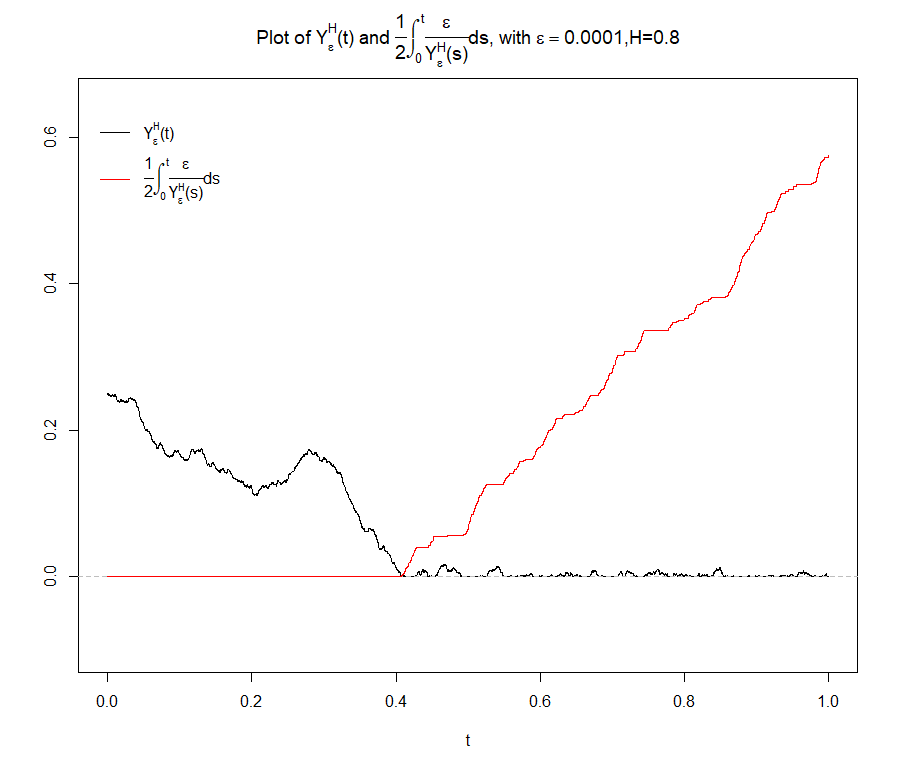}
(c) $H=0.8$
\end{minipage}
\begin{minipage}[b]{0.48\textwidth} \centering
\includegraphics[width=\textwidth]{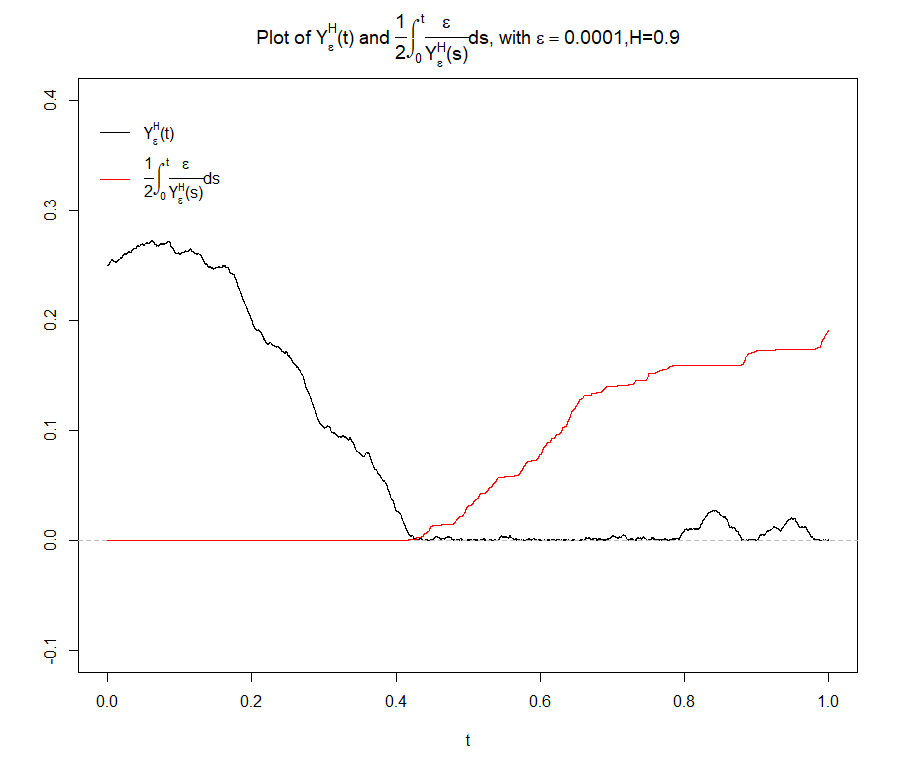}
(d) $H=0.9$
\end{minipage}
\caption{ Sample paths of $Y^H_{\varepsilon}(t)$ (black) and $\frac{1}{2}\int_0^t \frac{\varepsilon}{Y^H_\varepsilon(s)}ds$ (red) for $\varepsilon = 0.0001$ and different Hurst indices $H$.}\label{Fig1}
\end{figure}

Theorem \ref{th: representation of FROU process} states that the red line approximates the reflection function $L^H$ of the RFOU process and it can be clearly seen that the plot is well agreed with the theory: the integral $\frac{1}{2}\int_0^t \frac{\varepsilon}{Y^H_\varepsilon(s)}ds$ shows notable growth only when the corresponding path of $Y^H_\varepsilon$ is very close to zero.

Fig. \ref{Fig2} illustrates the uniform convergence of paths of $Y^H_\varepsilon$ to the path of RFOU process as $\varepsilon \downarrow 0$. On the picture, $H=0.6$, $Y(0) = 0.25$, $b=1$, $\sigma = 1$ and the path of the FROU process $Y^H$ was simulated using the Euler-type method:
\begin{gather*}
    Y^H(0) = Y(0),
    \\
    Y^H(t_{n+1}) = \max\left\{0,~Y^H(t_n) - \frac{b}{2}Y^H(t_n)(t_{n+1} - t_n) + \frac{\sigma}{2} \left(B^H(t_{n+1}) - B^H(t_n)\right)\right\}.
\end{gather*}

When $\varepsilon = 0.0001$, the path of $Y^H_\varepsilon$ (purple) is so close to the corresponding path of the ROU process (bold black) that they are not distinguishable on the plot.

\begin{figure}[h!]
  \centering
  \includegraphics[width=\textwidth]{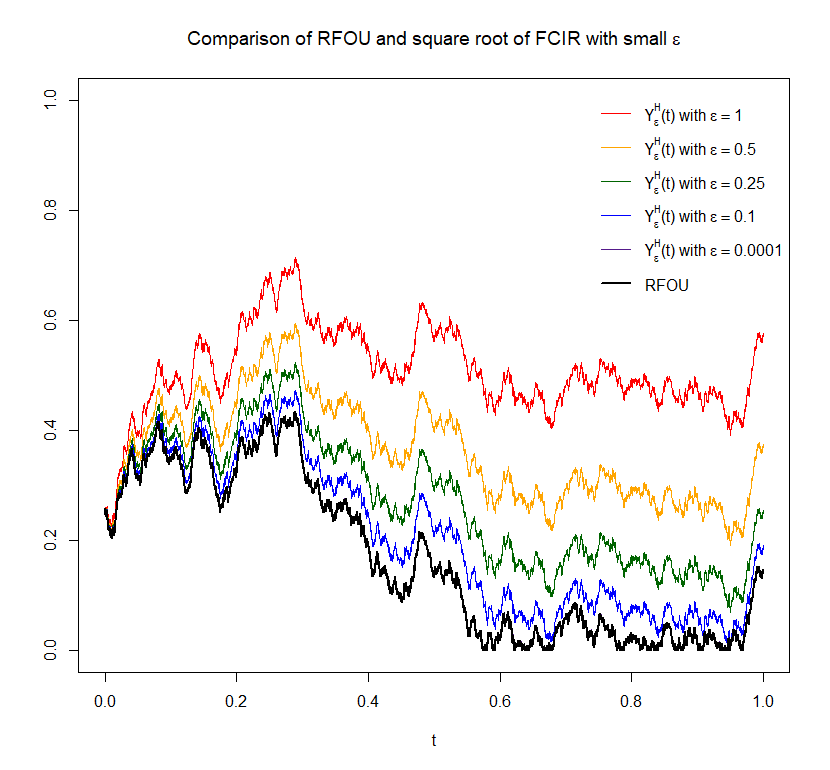}
  \caption{Comparison of the $Y^H_\varepsilon$ with $\varepsilon = 1$ (red), $\varepsilon = 0.5$ (orange), $\varepsilon = 0.25$ (green) $\varepsilon = 0.1$ (blue), $\varepsilon = 0.0001$ (purple) and the RFOU process (bold black). Note that the purple path ($\varepsilon = 0.0001$) is not visible on the plot since it almost completely coincides with the bold black trajectory of the RFOU process.}\label{Fig2}
\end{figure}

\section*{Acknowledgements}

The present research is carried out within the frame and support of the ToppForsk project nr. 274410 of the Research Council of Norway with title STORM: Stochastics for Time-Space Risk Models. The first author is supported by the National Research Fund of Ukraine under grant 2020.02/0026.

\bibliographystyle{tfs}
\bibliography{interacttfssample}

\end{document}